\documentclass[a4paper, 11 pt, twoside]{amsart}
\usepackage{srollens-en}[2011/03/03]
\usepackage[left = 3.5cm, right = 3.5cm, headsep = 6mm,
footskip = 10mm, top = 35mm, bottom = 35mm, footnotesep=5mm, headheight =
2cm]{geometry}
\usepackage{amsmath,amssymb,amscd,amsfonts}

\usepackage[usenames, dvipsnames]{xcolor}

\usepackage[T1]{fontenc}

\usepackage{tikz-cd}
\usepackage{booktabs, multirow}
\usepackage{enumitem}
\setlist[description]{leftmargin=0cm,  labelindent=\parindent}

\usepackage[unicode,bookmarks, pdftex, backref = page]{hyperref}
\hypersetup{colorlinks=true,citecolor=NavyBlue,linkcolor=NavyBlue,urlcolor=Orange, pdfpagemode=UseNone, breaklinks=true}

\renewcommand{\tilde}{\widetilde}
\newcommand{\pp}{\mathbb P}
\newcommand{\wt}{\widetilde}
\newcommand{\inv}{^{-1}}
\newcommand{\I}{\mathrm{i}}

\title{$(p,d)$-elliptic  curves of genus two}
\author{Marco Franciosi}
\address{Dipartimento di Matematica, Universit\`a di Pisa , Largo B. Pontecorvo 5, I-56127  Pisa, Italy}
\email{marco.franciosi@unipi.it}

\author{Rita Pardini}
\address{Dipartimento di Matematica, Universit\`a di Pisa , Largo B. Pontecorvo 5, I-56127  Pisa, Italy}
\email{rita.pardini@unipi.it}

\author{S\"onke Rollenske}
\address{S\"onke Rollenske\\FB 12/Mathematik und Informatik\\
Philipps-Universit\"at Marburg\\
Hans-Meerwein-Str. 6\\
35032 Marburg\\
Germany}
\email{rollenske@mathematik.uni-marburg.de}

\begin{document}
\begin{abstract}
We study  stable curves of arithmetic genus 2 which admit two morphisms of  finite degree $p$, resp. $d$, onto smooth elliptic curves,
with particular attention to the case $p$ prime.
\end{abstract}
\maketitle

\rightline{\textsc{\scriptsize Preliminary version \today}}
\maketitle

\setcounter{tocdepth}{1}
\tableofcontents

\section{Introduction}

In this paper we consider  stable curves of arithmetic genus 2 which admit a  $(p,d)$-elliptic configurations, namely two morphisms of  finite degree $p$, resp. $d$, onto smooth elliptic curves $D$ and $E$:
\begin{equation*}
  \begin{tikzcd}
   {} &C \arrow{dl}{p:1}[swap]{f}\arrow{dr}{g}[swap]{d:1}\\
   E && D
  \end{tikzcd}.
 \end{equation*} 
A  curve of genus $2$ is called $(p,d)$-elliptic curve if it admits a  $(p,d)$-elliptic configuration. If $p=2, d=3$ we use the terminology bi-tri-elliptic curve.

Frey and Kani in  \cite{frey-kani}  studied  genus 2 coverings of elliptic curves, i.e., genus 2  $d$-elliptic curves.  
A fundamental tool in  their research was the observation that given a degree $d$ map $f\colon C\to E$ onto an elliptic curve then there exists a  
complementary  map $f ' \colon C\to E'$ of the same degree $d$  onto a second elliptic curve (see \S  \ref{sec:d-ell}). 
Later Kani in \cite{kani-number}  studied in details the arithmetic properties  of such $(d,d)$-elliptic configurations, also  providing existence results.

In this paper we study $(p,d)$-elliptic  configurations  following the construction described  by  Frey and Kani in  \cite{frey-kani}. 

It  will turn  out that stable $(p,d)$-elliptic curves of arithmetic genus $2$ are automatically of compact type, i.e., they have  compact Jacobian; thus in the first section we recall  the Frey-Kani construction, noting that it extends  to curves of compact type.

In Section \ref{section:(p,d)} we study $(p,d)$-elliptic curves, with particular  attention to the case $p$ prime. In Theorem \ref{thm:2d} we give a classification of such curves and in Section \ref{sec:existence} we show   that for every pair of integers  $p,d>1$  there exists a smooth $(p,d)$-elliptic curve of genus 2.

The original motivation for this article was the study of bi-tri-elliptic configurations, which parametrise certain strata in the boundary of the moduli space of stable Godeaux surfaces (see \cite{FPR16b}). Thus we describe the geometry of bi-tri-elliptic configurations in a little more detail in the last section.

\subsubsection*{Acknowledgements}
We wish to thank Ernst Kani and Angelo Vistoli for useful mathematical communications. 
  The second  author  is a  member of GNSAGA of INDAM. The third author is grateful for support of the DFG through the Emmy Noether program and partially through SFB 701. This project was partially supported by PRIN 2010 ``Geometria delle Variet\`a Algebriche'' of italian MIUR.

\section{$d$-elliptic curves of genus 2}\label{sec:d-ell}
Here we  recall and slightly refine some results from \cite[\S 1]{frey-kani}, where the focus is on smooth curves and on  the case $d$ odd (see below).
\subsection{Set-up and preliminaries}
We   work over an algebraically closed    field $\IK$ whose  characteristic  does not divide the degree $d$ of the finite morphisms that we consider. Throughout all this  section $C$ is a stable curve of genus $2$ and 
$J = J(C)$ is  the Jacobian of $C$. 

\begin{defin} \label{def:d-ell}
 Let $d\ge 2$ be an integer. We say that $C$ is {\em $d$-elliptic} if there exists a finite degree $d$ morphism $f\colon C\to E$ such that $E$ is a smooth curve of genus 1 and $f$ does not factor through an \'etale cover of $E$; we call $f$ a  {\em $d$-elliptic map}. 
Sometimes,  a $d$-elliptic map is      called  an  ``elliptic subcover'' and the curve $C$ is said to have an ``elliptic differential'' (cf. \cite{kani-number}); our choice of terminology is due to the fact that we wish to emphasize  the degree $d$ of the map.
 For $d=2,3$, the curve $C$ is also called {\em bi-elliptic}, resp. {\em tri-elliptic}. 
 
  An isomorphism of $d$-elliptic curves $f_i\colon C_i\to E_i$, $i=1,2$,  is a pair of  isomorphisms $\phi\colon C_1\to C_2$  and $\bar\phi\colon E_1\to E_2$ such that  $f_2\circ \phi =\bar \phi\circ f_1$.
\end{defin}
  
 For an abelian variety $A$ we denote by $A[d]$ its subgroup of $d$-torsion points.
If $A$ is principally polarised then there is a non degenerate  alternating pairing $e_d\colon A[d]\times A[d] \to \mu_d$ 
(where $\mu_d$  denotes the $d$-th roots  of unity)
called Weil pairing (or Riemann form \cite[Ch.IV,~\S20]{Mumford_Abelian}).

If $A'$ is an  abelian variety,  then we call a group homomorphism $\alpha\colon A[d]\to A'[d]$ \emph{anti-symplectic} if  for every $P,Q\in A[d]$ one has:
$$e_d(\alpha(P),\alpha(Q))=e_d(P,Q)\inv,$$
or, equivalently, if the graph of $\alpha$ is an isotropic subgroup of $(A\times A')[d]$.

\subsection{The Frey-Kani construction}\label{Frey-Kani construction}
Now assume that  $C$ is a stable genus 2 curve of compact type, 
i.e., it is either smooth or the union of two elliptic curves intersecting in one point. Notice that  the Jacobian  $J = J(C)$  is a principally polarised abelian surface. 

Let 
$f\colon C\to E$ be a $d$-elliptic map on $C$. 
 The pull back map $f^*\colon E\to J$ is injective, hence  the norm map $f_*\colon J\to E$ has connected kernel $E'$.  Since the composition $f_*f^*\colon E \to E$ is  multiplication by $d$, the abelian subvarieties  $E'$ and $f^*E$ of $J$  intersect in $E[d]$ and we have a tower of isogenies
  \begin{equation}\label{eq: tower of isogenies}\begin{tikzcd} E\times E' \rar{h }[swap]{d^2:1} & J \rar{h'}[swap]{d^2:1}  & E\times E'\end{tikzcd},
  \end{equation}
   whose composition is  multiplication by $d$. Composing the Abel-Jacobi map $C\into J$ with the projection to $E'$ we get a second $d$-elliptic map $f'\colon C\to E'$, which we call the       
\emph{complementary}  $d$-elliptic map. 

The construction that follows, which we call the Frey-Kani construction,  has been described  in \cite[\S 1]{frey-kani} for smooth curves, but the proof works verbatim for stable curves of compact type.  Therefore one has:  
\begin{prop}\label{prop:Jd}
Let $C$ be a stable $d$-elliptic  curve of genus $2$  of compact type,  
let  $f\colon C\to E$ and $f'\colon C\to E'$ be  complementary $d$-elliptic maps and let  $h\colon E\times E'\to J=J(C)$ be as in \eqref{eq: tower of isogenies}. Then: 
 \begin{enumerate}
 \item there exists an anti-symplectic  isomorphism $\alpha \colon E[d]\to E'[d]$ such that $\ker h$ is the graph $H_{\alpha}$ of $\alpha$;
 \item the principal polarization on $J$ pulls back to $d(E\times\{0\}+\{0\}\times E').$
\end{enumerate}
\end{prop}
Notice that if $d=2$, then any  isomorphism $\alpha$  as in Proposition \ref{prop:Jd} is anti-symplectic. More generally, for a prime $d$   the number of anti-symplectic isomorphisms $E[d]\to E'[d]$ is equal to $d(d^2-1)$ (cf. \cite{frey-kani}).

The above proposition has a converse (see  \cite{frey-kani}):
\begin{prop} \label{prop:dJ}
Let $E,E'$ be elliptic curves and   let $\alpha\colon E[d]\to E'[d]$ be an anti-symplectic isomorphism.  Denote by   $H_{\alpha}$ the graph of $\alpha$; set $A:=(E\times E')/H_{\alpha}$ and denote by $h\colon E\times E'\to A$ the quotient map. 

Then 
 \begin{enumerate}
 \item $d(E\times\{0\}+\{0\}\times E')$ descends to a principal polarization $\Theta$   on $A$;
 \item let $C$ be a Theta-divisor on $A$; then $C$ is a stable  curve of genus 2 of compact type and the maps $f\colon C\to E$ and $f'\colon C\to E'$ induced by the natural maps $A\to E$ and $A\to E'$ are complementary $d$-elliptic maps. 
  \item if $d$ is odd, then there is precisely one symmetric Theta-divisor on $A$  that  is linearly equivalent to $d(E\times\{0\}+\{0\}\times E')$
\end{enumerate}
\end{prop}

\subsection{Special geometry for small $d$}
It is an interesting question under what conditions the polarisation coming from the Frey-Kani construction is reducible. 

We answer this question for $d=2$.
\begin{lem}\label{lem:bi-red} Let $A$ be constructed as in  Proposition \ref{prop:dJ}   for $d=2$. Then the principal polarization $\Theta$  of $A$ is reducible if and only if there exists an isomorphism $\psi\colon E'\to E$ such that the map $ E\times E'\to E\times E$ defined by $(x,y)\mapsto (x,\psi(y))$ maps $H_{\alpha}$ to the subgroup $\Delta[2]=\{(\eta,\eta)|\eta \in E[2]\}$.

Moreover, up to isomorphism the bi-elliptic map $f$ is given by the composition
\[ C = E\times\{0\}\cup\{0\}\times E \into J(C) =  E\times E \overset{+}{\longrightarrow} E,\]
that is, it is  the identity on each component of $C$;   the complementary map $f'$ is the identity on one component and multiplication by $-1$ on the other. 
\end{lem}
\begin{proof}
Denote by $\bar E$, $\bar E'$ the image of $E$, $E'$ in $A$:  we have $\Theta=(\bar E+\bar E')/2$ and $\bar E \bar E'=4$. Assume  that $\Theta=B_1+B_2$ is reducible, with $B_1$, $B_2$ elliptic curves such that $B_1B_2=1$ (recall that in this case $A\cong B_1\times B_2$ as ppav's). Then $1=B_1\Theta=B_1(\bar E+\bar E')/2$ implies either $B_1\bar E=B_1\bar E'=1$ or, say, $B_1 \bar E=2$, $B_1\bar E'=0$. The latter case cannot occur, since we would have $B_1=\bar E'$, hence   $2=B_1\bar E=\bar E'\bar E=4$, a contradiction. 

Hence composing the map $E\to \bar E$ with the projections  $A \to B_i$, for $i=1,2$ one obtains isomorphisms $\phi_i\colon E\to B_i$.
Hence we have an isomorphism $A\to E\times E$ that maps $\bar E$ to  the diagonal in $E\times E$ and $\bar E'$ to  the graph of an automorphism $\sigma$ of $E$. Since $\bar E\bar E'=4$,  $\sigma$ has 4 fixed points. By the classification of the possible  automorphism groups of an elliptic curve, it follows that $\sigma$ is  multiplication by $-1$ and  $\bar E'$  is mapped to the antidiagonal. Composing $E'\to A$ with the isomorphisms $A\to E\times E$ and with the first projection $E\times E\to E$ we obtain the required  isomorphism $\psi\colon E'\to E$.
Indeed it is not hard to check that the following diagram commutes:
\begin{equation}\label{diag: product reducible polarisation}
\begin{tikzcd}
E\times E' \rar{(\id, \psi)}\dar{h}& E\times E\dar{q}\\
A \rar & E\times E
\end{tikzcd},
\end{equation}
where $h$ is the quotient map and $q(x,y)=(x+y,x-y)$. 

Conversely, assume that $E=E'$ and $\alpha$ is the identity. The map $q\colon E\times E\to E\times E$ defined by $q(x,y)=(x+y,x-y)$ has kernel $H_{\alpha} = \Delta[2]$, hence $A$ is isomorphic to $E\times E$. Let $C=E\times \{0\}+\{0\}\times E$; then $q^*C=\Delta+\Delta^{-}$, where $\Delta$ is the diagonal and $\Delta^-$ is the antidiagonal. 
Since $(q^*C)^2=8$ by the pull-back formula and $q^*C(E\times \{0\}+\{0\}\times E)=4$,  the divisors $q^*C$ and $2(E\times \{0\}+\{0\}\times E)$ are algebraically equivalent by the Index Theorem, hence $C$ is the principal polarisation of Proposition \ref{prop:dJ}. (More precisely, since $q^*C$ and $2(E\times \{0\}+\{0\}\times E)$ restrict to the same divisor on $E\times \{0\}$ and $\{0\}\times E$ they are actually linearly equivalent). The final part of the statement follows.
\end{proof}

We close this section with an alternative description of bi-elliptic curves of genus $2$ of compact type, which basically stems from the fact that a double cover is the quotient by an involution.

\begin{lem}
 Let $C$ be a genus 2   stable  curve of compact type,  let $f\colon C\to E$ and $f'\colon C\to E'$ be complementary bi-elliptic maps and  let $\sigma$, resp. $\sigma'$, be  the involution induced by $f$, resp. $f'$. 
 Then the group  $\langle  \sigma, \sigma' \rangle$ is isomorphic to $  \IZ/2\times \IZ/2  $ and $ \tau:= \sigma  \sigma'  $  is the hyperelliptic involution.
\end{lem}
\begin{proof}  Let $J$ be the Jacobian of $C$.  The involution $\sigma$ on $J$ is induced by the  involution $(x,y)\mapsto (x,-y)$ of $E\times E'$ (cf. \eqref{eq: tower of isogenies}) and, similarly, $\sigma'$ is induced by $(x,y)\mapsto (-x,y)$. So $\tau=\sigma\sigma'$ acts as multiplication by $-1$ on $J$ and therefore, if $C$ is smooth, is the hyperelliptic involution. If $C$ is reducible, then $\tau$ is multiplication by $-1$ on both components of $C$. 

Since $\tau$ is in the center of $\Aut(C)$, $\sigma$ and $\sigma'$ commute and $\langle  \sigma, \sigma' \rangle$ has order 4.
\end{proof}
 
It is well known that bidouble covers can be reconstructed from their building data (cf. \cite[\S 2]{pardini91});  we explain this in the case at hand.

  Choose points  $P_1,P_2$ and  distinct points $Q_1,Q_2,Q_3\in \pp^1$ that are also distinct from $P_1$ and $P_2$ ($P_1$ and $P_2$ are allowed to  coincide). Let $\pi\colon C\to \pp^1$ be the bidouble cover branched on $D_1=P_2$, $D_2=P_1$ and $D_3=Q_1+Q_2+Q_3$,  denote by $G$ the Galois group  of $\pi$ and  by $\sigma\in G$
  (resp. $\sigma '$ and $\tau$)  the involution  the fixes the preimage of $D_1 $ (resp. $D_2$,  $D_3$). 
  Assume first that $P_1\ne P_2$; in this case $C$ is a smooth curve of genus 2 and for $i=1,2$ the quotients   $E=C/\sigma $ and  $E'=C/\sigma'$ are  smooth curves of genus $1$. The involution $\tau$ has $6$ fixed points and therefore  is the hyperelliptic involution. 

If $P_1=P_2$, then $C$ has a node over $P_1=P_2$ and the normalization is the bidouble cover of $\pp^1$ with branch divisors $D_1=D_2=0$ and $D_3=P_1+Q_1+Q_2+Q_3$. So $C$ is reducible and has  two components, both isomorphic to the double cover of  $\pp^1$ branched on $P_1+Q_1+Q_2+Q_3$.

This construction  is related to the construction given in Proposition  \ref{prop:dJ}   as follows. 

Let $\pi\colon C\to \pp^1$ be as above, with $C$ smooth,  and 
 take the preimage of $P_1$ as the origin  $0\in E$ and the preimage of $P_2$  as the origin  $0'\in E'$. Denote by $A_1, A_2, A_3$ (resp. by $B_1, B_2, B_3$) the preimages of $Q_1,Q_2,Q_3$ in $E$ (resp.  in $E'$). 
Then the nonzero elements of $E[2]$ (resp. $E'[2]$)  are $\eta_i:=A_i-0$ (resp. $\eta'_i:=B_i-0'$), $i=1,2,3$); we define $\alpha\colon E[2]\to E'[2]$ as the isomorphism that maps $\eta_i\to \eta'_i$. 

We claim that the bi-elliptic structure on $C$ is obtained via  the Frey-Kani construction with the above choice of   $\alpha$, i.e.,  the 
  kernel of the  pull-back map $h'^{\ast}\colon E\times E'\to J:=J(C)$ is the graph $H_{\alpha}$ of $\alpha$. 

Indeed, since the kernel $\Gamma$  of the pull-back map has order 4, it is enough to show that $H_{\alpha}$ is in contained in $\Gamma$. In addition one has $f^*A_i={f'}^*B_i$ for $i=1,2,3$, hence we only need to show that  $f^*0$ and ${f'}^*0'$ are linearly equivalent. The divisor $f^*0$ is the ramification divisor of $f'$, hence $f^*0\equiv K_C$; the same argument shows that ${f'}^*0'\equiv K_C$ and we are done. 

The case $C$ reducible is obvious.

\begin{rem}
 For tri-elliptic curves one can apply the general theory of triple covers \cite{miranda85} to deduce the following result \cite[Lem.~2.8]{FPR16b}: a stable curve $C$ of genus $2$ admits a tri-elliptic map  $C\to E$ such that $C$ embedds into the symmetric square of $E$ as a tri-section of the Albanese map $S^2E \to E$. 
 
 Note however that a tri-elliptic map $C\to E$ cannot be a cyclic cover, since by the Hurwitz formula it would be ramified over precisely one point and this is impossible, for instance, by   \cite[Prop. 2.1]{pardini91}).  So there is no elementary description of $C$ just in terms of the ramification divisor. 
\end{rem}

\section{$(p,d)$-elliptic curves of genus 2}\label{section:(p,d)}
We consider stable curves of genus 2 admitting two distinct maps to elliptic curves.

\subsection{$(n,d)$-elliptic curves and configurations}

\begin{defin}
Let $C$ be a stable curve of genus $2$. 
An $(n,d)$-elliptic configuration $(C,f,g)$  is a diagram
\begin{equation}\label{diag: pd-elliptic}
  \begin{tikzcd}
   {} &C \arrow{dl}{n:1}[swap]{f}\arrow{dr}{g}[swap]{d:1}\\
   E && D
  \end{tikzcd},
 \end{equation}
 where $f$ is an $n$-elliptic map and $g$ is a $d$-elliptic map such that  there is no isomorphism  $\psi \colon E\to D$ such that $g=\psi\circ f$. 
 We refer to $C$  as to an $(n,d)$-elliptic curve (of genus 2).
\end{defin}
An isomorphism of $(n,d)$-elliptic configurations is an isomorphism of diagrams like \eqref{diag: pd-elliptic}.

\begin{lem}\label{lem:compact type} If  $C$ is $(n ,d)$-elliptic stable curve of genus 2, then it is of compact type. 
\end{lem}
\begin{proof}
As usual, we let $f\colon C\to E$ and $g\colon C\to D$ be the $n$-elliptic and $d$-elliptic maps.

The curve $C$ cannot have rational components since it has a finite map onto an  elliptic curve, hence it is enough to rule out  the possibility that $C$ is irreducible with  one node. 
Assume by contradiction that this is the case, let $C^{\nu}\to C$ be the normalization map and  denote by $\psi_1\colon C^{\nu}\to E$ (resp.  $\psi_2\colon C^{\nu}\to D$) 
the map of degree $n$ (resp. $d$)    induced by  $f$ (resp. $g$).  
We denote by $O,P\in C^{\nu}$ the points of  that map to the node of $C$ and we take $O$ as the origin; with a suitable choice of the origins in  $E$ and $D$  we can assume that $\psi_1$ and $\psi_2$ are isogenies.

  Since $\psi_1$ and $\psi_2$ factor through $C^{\nu}\to C$, $P$ belongs to $\ker \psi_1\cap \ker\psi_2$ and for $i=1,2$ $\psi_i$  factors through the \'etale covers $C^{\nu}/\!<P>\to E$ and $C^{\nu}/\!<P>\to D$. It follows that $f$ and $g$ also factor through $C^{\nu}/\!<P>\to E$ and $C^{\nu}/\!<P>\to D$ hence,  
   by the definition of   $d$-elliptic curve,  it follows that $n =d=2$ and the two bi-elliptic structures differ by an isomorphism $E\to D$, a contradiction. \end{proof}

\begin{rem}\label{rem:d-d}
By the Frey-Kani construction given in  \S \ref{sec:d-ell}, if $C$ is of compact type and has a $d$-elliptic map $f\colon C\to E$ then, if  we denote  $f'\colon C\to E'$  complementary  $d$-elliptic map, $(C, f, f')$ is a  $(d,d)$-elliptic configuration. We refer to this as to  the \emph{trivial $(d,d)$-elliptic configuration}. 
\end{rem}

\subsection{$(p,d)$-elliptic curves}\label{ssec: pd}

Assume we are given an $(n,d)$-elliptic configuration as in \eqref{diag: pd-elliptic}, which is non-trivial in the sense of Remark \ref{rem:d-d}. Then both elliptic maps factor,
up to isomorphism, through the Abel-Jacobi map of $C$
 and are thus uniquely determined by the subgroups $\ker f_*$ and  $\ker g_*$.

 To analyse such a configuration,  we apply the Frey-Kani construction to the $n$-elliptic map  given by $f$ and then  we consider $\bar F = \ker g_*$.  We can   extend \eqref{eq: tower of isogenies} to a diagram 
\begin{equation}\label{eq: diagrammone}
 \begin{tikzcd}
  F\dar[dashed][swap]{(\phi, \phi')}\rar[dashed]{h_F}& \bar F\dar\arrow{drr}{m:1} \\
  E\times E' \rar{h} &  J(C) \rar{f_*\times f'_*}\dar{g_*} & E\times E' \rar & E\\
   & D & 
 \end{tikzcd},
\end{equation}
where $\bar F \Theta = d$    ($\Theta$ denoting the principal polarization) and $F$ is the connected component of $h\inv F$ containing the origin.
Indeed  we have the following

\begin{rem}\label{rem:key} 
A  genus 2 curve of compact type has an $n$-elliptic structure if and only if its Jacobian $J$ contains a connected $1$-dimensional subgroup $\bar{E}'$ such that $\bar{E}' \Theta=n$.
 Therefore an  $n$-elliptic curve  
 $C$ has an $(n,d)$-elliptic structure if and only if $J$ contains a  second  connected 1-dimensional subgroup $\bar F$ such that $\bar F \Theta=d$ and $ \bar F \ne \bar{E}'$. 
 So, if  we exclude the trivial $(d,d)$-structures (cf. Remark \ref{rem:d-d}), the Jacobian of an  $(n,d)$-elliptic curve of genus 2 contains at least three, hence infinitely many, connected $1$-dimensional subgroups. In particular  the curve $C$ has infinitely many elliptic structures, and  the curves  $E$ and $E'$   are isogenous.
 \end{rem}

For a given      $(n,d)$-elliptic configuration  $(C,f,g)$ we denote  $\bar F = \ker g_*$ and  $\bar{E}'= \ker f_*$ and 
we define the \emph{twisting number}  of  $(C,f,g)$ as 
\begin{equation}\label{eq: m} m = m(C, f, g) := \bar F \ker f_*=  \bar F \bar{E}'  = \deg (\bar F\to E ) = n^2\frac{\deg \phi}{\deg h_F}
\end{equation}   
where $\phi$ and $h_F$ are as in \eqref{eq: diagrammone}. 
\begin{rem} \label{rem:trivial} 
One has $m>0$, by the definition of $(n,d)$-elliptic curve. 

Denote by $\bar E$ the kernel of $f'_*\colon J\to E'$, where $f'$ is the complementary map of $f$. Then by the Frey-Kani construction we have $\Theta= \frac{\bar E+ \bar E'}{n}$, hence $nd=n\bar F\Theta=m+\bar F\bar E$. It follows that $m\le nd$, with equality holding if and only if $\bar F\bar E=0$, namely iff  we are in the trivial case $g=f'$.
\end{rem} 
                                                                 
We   first provide three examples with $m <nd$  that fit into this general pattern and then prove that when $n=p$ is a prime these cover all possibilities for non trivial elliptic configurations.

\begin{exam}\label{ex:1}
Let  $n,d$ be integers. 
Let $F$ be an elliptic curve and let $\phi\colon F\to E$ and $\phi' \colon F\to E'$ be isogenies such that: 
\begin{itemize}
\item[--]$\ker \phi\cap \ker \phi'=\{0\}$, hence  $\phi\times\phi'\colon F\to E\times E'$ is injective; 
\item[--] $\deg\phi+\deg \phi'=nd$, and $n$  and $\deg\phi$ are coprime.
\end{itemize}

We abuse notation and denote again by $F$ the image of $\phi \times \phi'$. 
The subgroup $H:=F[n]\subset (E\times E')[n]$ satisfies  $H\cap E=H\cap E'=\{0\}$,  since $EF=\deg \phi'$ and $E'F=\deg\phi$ are  coprime to  $n$. Hence $H$ is the graph of an isomorphism $E[n]\to E'[n]$. 
The polarization $n(E\times\{0\}+\{0\}\times E')$ restricts on $F$ to a divisor of degree $n^2d$, which therefore is a pull back via the map $F\to F$ defined by multiplication by $n$. 
By the functorial properties  of the  
 Weil pairing (see  statement (1) of  \cite[Ch.IV,~\S23, p.228]{Mumford_Abelian})  it follows   that $F[n]$ is an isotropic subspace of $(E\times E')[n]$. 

Let $A =(E\times E')/H$ and let   $\Theta$ be the principal polarization of $A$ (see   Proposition \ref{prop:dJ}). 
  Denote by  $\bar F$ be the image of $F$ in  $A$: 
 then  we have $n^2\bar F\Theta=nF(E\times \{0\}+\{0\}\times E')=n^2d$, namely $\bar F\Theta=d$.  By Remark \ref{rem:key}   
we obtain an $(n,d)$-elliptic configuration
with twisting number $m = n^2\frac{\deg \phi}{\deg h_F}= \deg \phi$ (cf.   \eqref{eq: m}). 
 \end{exam} 
 
\hfill\break
From now on we will focus on the case where $n=p$ is a prime number. 

\begin{exam}\label{ex:p}
Let  $p,d$ be integers and 
assume that $p$ is a prime. Let $F$ be an elliptic curve and let   $\phi\colon F\to E$ and $\phi' \colon F\to E'$ be isogenies such that: 
\begin{itemize}
\item[--] $\ker \phi\cap \ker \phi'=\{0\}$;
\item[--] $\deg\phi+\deg \phi'=d$;
\item[--]  $F[p]\not\subset \ker \phi$ and $F[p]\not \subset \ker \phi'$.  
\end{itemize}

Under the above conditions,  it is possible to  find  an antisymplectic isomorphism $\alpha\colon E[p]\to E'[p]$ such that  $H_{\alpha}\cap F$ has order p, where $H_{\alpha}$ is the graph of   $\alpha$.  This follows  because   by 
our assumptions there 
exists  $0\ne v\in F[p]$ such that $v\notin \ker\phi \cup \ker \phi'$. 
Moreover,
 since  the Weil pairing of a product is given by the product of the Weil pairings  (see   statement (2) of  \cite[Ch.IV,~\S23, p.228]{Mumford_Abelian}), the annihilator $W$ of $v$ in $(E\times E')[p]$ does not contain $E[p]\times \{0\}$ nor   $\{0\}\times E'[p]$. 
 The linear subspace  $W$ is three dimensional, hence $\pp(W)$ is a projective plane over $\IF_p$.
 Now consider in $\pp(W)$  the pencil  $\mathcal F$ of lines through $[v]$: since $\mathcal F$  consists of $p+1$ lines,   if $p>2$ there is at least a line $l \in  \mathcal F$  that does not intersect the lines $r:=\pp(E[p]\times \{0\})$ and $s:=\pp(\{0\}\times E'[p])$ and distinct from $t:=\pp(F[p])$. 
 The subspace of $(E\times E')[p]$  corresponding to $l$ is the graph of an isomorphism $\alpha\colon E[p]\to E'[p]$ and is also isotropic, hence $\alpha$ is anti-symplectic.
 For $p=2$, any isomorphism $\alpha$ is antisymplectic, hence it is enough to find a line in $\pp((E\times E')[2])$ that contains $[v]$,  which is distinct from $t$  and does not intersect $r$ and $s$.  An elementary geometrical  argument shows that there exist two lines with this property. 

Therefore  we can consider $ A:=(E\times E')/H_{\alpha}$ and  the principal polarization $\Theta$  of $A$ (cf.   Proposition \ref{prop:dJ}). 
 Again we denote by $\bar F$ the image of $F$ in $ A$, obtaining 
 $p^2\bar F\Theta=pF(p(E\times \{0\}+\{0\}\times E'))=p^2d$, namely $\bar F\Theta=d$,  
 i.e. by Remark \ref{rem:key}    we get a  $(p,d)$-elliptic configuration. In this case by  \eqref{eq: m}  we have  $m = p^2\frac{\deg \phi}{\deg h_F}=  p \deg \phi$.  
\end{exam}

\begin{exam}\label{ex:p2}
Let  $p,d$ be integers such that   
 $p$ is a prime  and  $d$ is divisible by $p$.  Write $d=p\delta$ and 
let $F$ be an elliptic curve with   $\phi\colon F\to E$ and $\phi' \colon F\to E'$  isogenies such that:
\begin{itemize}
\item[--] $\ker \phi\cap \ker \phi'=\{0\}$;
\item[--] $\deg\phi+\deg \phi'=\delta$;
\end{itemize}
We look for an antisymplectic isomorphism $\alpha\colon E[p]\to E'[p]$ such that $H_{\alpha}\cap F=\{0\}$, $H_{\alpha}$ being  the graph of $\alpha$. 

To see that such $\alpha$ exists we  argue as follows.  
As in Example \ref{ex:p} we  identify $2$-dimensional  subspaces of $(E\times E')[p]$ with lines in $\pp^3({\mathbb F}_p):=\pp((E\times E')[p])$. We have seen in Example \ref{ex:p} that there are $p+1$ isotropic lines through any  point, hence there  exist  $(p+1)(p^2+1)$ isotropic lines.

The isotropic lines meeting a given line   are  $p(p+1)+1$  or  $(p+1)^2$, according to whether the line is isotropic or not. Set  $r:=\pp(E[p]\times \{0\})$,   $s:=\pp(\{0\}\times E'[p])$ and  $t=\pp(F[p])$; note that $r$ and $s$ are not   isotropic. 
Hence there are at most $3(p+1)^2$ isotropic lines meeting $r\cup s\cup t$.  However, all the lines joining a point of $r$ and a point of $s$ are isotropic hence, by  subtracting these lines (that we had counted twice) we get the  better 
upper estimate   $ 2(p+1)^2$ for the number of isotropic lines   meeting $r\cup s\cup t$.  For $p\ge 3$ this shows the existence of the isotropic subspace $H_{\alpha}$ that we are looking for, since $(p+1)(p^2+1)-2(p+1)^2=(p+1)(p^2-2p-1)>0$.

For $p=2$, we need only find a line that is disjoint from $r\cup t \cup s$. We observe that $\pp^3({\mathbb F}_2)$ contains 35 lines,  that the lines intersecting a given line are 19,  that the lines meeting two given skew lines are 9 and the lines meeting three mutually skew lines are 3.

 If  $\deg\phi$ and $\deg \phi'$ are odd,  then the  three lines $r$,  $s$ and $t$ are mutually skew: then the set of lines meeting at least one of these consists of  $3\cdot 19- 3\cdot 9+3= 33$ lines, hence there are 2 possibilities for $H_{\alpha}$. 
 
 Now assume that both  $\deg \phi$ and $\deg\phi'$ are even. In this case $t$ meets both $r$ and $s$. The number of lines intersecting $r\cup t$ is equal to $7+7-3=11$, since there are 7 lines in  plane spanned by $r$ and $t$, there are  7 lines passing through $r\cap t$, and 3 lines common to these two sets. An analogous argument shows that there are $3+3-1=5$ lines meeting $t$, $r$ and $s$.
 So the number of lines intersecting $r\cup t \cup s$ is equal to $3\cdot 19-9-2\cdot 11+5=31$, so there are 4 possibilities for $H_{\alpha}$.

Finally we consider the case where $\deg \phi$ is even and $\deg\phi'$ is odd. There are two possibilities: either $r=t$ or $r$ and $t$ are coplanar but distinct. 
In the former case, the  number of lines  meeting $r\cup t \cup s=r\cup s$ is equal to $2\cdot 19-9=29$, so there are 6  possibilities for $H_{\alpha}$.
In the latter case,  the  number of lines  meeting $r\cup t \cup s$ is equal to $3\cdot 19-2\cdot 9-11+5=33$, so there are 2  possibilities for $H_{\alpha}$.

Taking $ A:=(E\times E')/H_{\alpha}$ and  $\Theta$   the principal polarization  of $A$ (cf.   Proposition \ref{prop:dJ}) and 
 denoting by $\bar F$ the image of $F$ in $ A$, we get 
 $p^2\bar F\Theta=p^2F(p(E\times \{0\}+\{0\}\times E'))=p^2d$, namely $\bar F\Theta=d$,  
 i.e. by Remark \ref{rem:key}  
we get a  $(p,d)$-elliptic configuration. In this case    \eqref{eq: m} yields  $m = p^2\frac{\deg \phi}{\deg h_F}=  p^2 \deg \phi$.   
\end{exam}

\begin{thm}\label{thm:2d}
Let $p$ be a prime and let $d$ be a positive integer. Let  $C$  be a stable curve of genus 2  and let $(C,f,g)$ be  a  non-trivial (cf. Remark \ref{rem:d-d}) $(p,d)$-elliptic configuration 
 \[  \begin{tikzcd}
   {} &C \arrow{dl}{p:1}[swap]{f}\arrow{dr}{g}[swap]{d:1}\\
   E && D
  \end{tikzcd},
\]
Denote by $\bar E'$ (resp. $\bar F$)  the kernel of  $f_*\colon J=J(C)\to E$  (resp. $g_*\colon J\to D$) and let $m=\bar E'\bar F$ be the twisting number as in  \eqref{eq: m}.
 Then
\begin{enumerate}
\item  the $(p,d)$-elliptic configuration  arises as in Example \ref{ex:1}, or \ref{ex:p}, or \ref{ex:p2}, and thus $1\leq m\leq pd -1$;
\item the case of Example \ref{ex:p2}  can occur  only if $p$ divides  $d$ and  $p^2$ divides $m$;
\item the case of  Example \ref{ex:1} occurs if and only if $m$ is not divisible by $p$. 
\end{enumerate}
\end{thm}

\begin{proof} By   Remark \ref{rem:trivial} we have $0<m\le pd$, and $m=pd$  holds only in the trivial case $g=f'$.
Therefore    by our assumptions it is  $1\le m\le pd-1$. 

We use freely the notation of \S \ref{ssec: pd} and diagram \ref{eq: diagrammone} and   we denote by $\phi\colon F\to E$ and $\phi'\colon F\to E'$ the isogenies induced by the two projections of $E\times E'$. Note that $\ker\phi\cap\ker \phi'=\{0\}$ by construction. 
The pull-back $h^*\bar F\subset E\times E'$ is algebraically equivalent to $\nu F$ for some integer $\nu\in\{1,p,p^2\}$ (one has $p^2=\nu|H\cap F|$). We have $p^2m=p^2\bar F\bar E'=\nu F(p^2(\{0\}\times E'))$, i.e., $m =\nu F(\{0\}\times E')=\nu\deg\phi$. In the same way, one obtains $pd-m=\nu F(E\times \{0\})=\nu \deg \phi'$.  In particular,  $\nu=1$ if $m$ is not divisible by $p$.

Consider the case $\nu=1$, i.e, $H=F[p]$. In this case  the map $E\times E'\to J(C)$ induces a degree $p^2$  isogeny $F\to \bar F\cong F$, the degree of $\phi$ is equal to $m$ and the degree of $\phi'$ is equal to $pd-m$.  
Since $H$, being a graph, intersects $E\times \{0\}$ and $\{0\}\times E'$ only in $0$, it follows that  $m$, which  is equal to the order of  $(\{0\}\times E')\cap F$,  is prime to $p$, and the same is true for  $\deg \phi'=pd-m$.
So, $C$ is constructed as in Example \ref{ex:1}.

Next, assume that $\nu=p$, i.e. $H\cap F$ has order $p$.  In this case, one has $m=p\deg \phi$ and  $\deg\phi+\deg \phi'=d$. Since $H\cap F$ has order $p$ and $H$ is a graph, it follows that $F[p]\not\subset \ker \phi$ and $F[p]\not \subset \ker \phi'$, hence  $C$ is constructed as in Example \ref{ex:p}.

Finally consider the case $\nu=p^2$.  In this case one has $m=p^2\deg \phi$ and  $d=p(\deg\phi+\deg \phi')$,  hence  $C$ is constructed as in Example \ref{ex:p2}.
\end{proof}

\subsection{Existence of smooth $(n,d)$-elliptic curves}\label{sec:existence} 

 First of all let us recall that  
 an irreducible $(n,d)$-elliptic curve is smooth by Lemma \ref{lem:compact type}.

By Lemma \ref{lem:bi-red},  for $n=2$ a necessary condition for the irreducibility of the genus 2 curve $C$ constructed as in Proposition \ref{prop:dJ} is that the curves $E$ and $E'$ are isomorphic, hence if $E$ does not have complex multiplication then the constructions of Examples \ref{ex:1}, \ref{ex:p} and \ref{ex:p2} yield examples of smooth $(2,d)$-elliptic curves of genus 2 for every $d>2$. 

In general,  it is not clear  whether the constructions of Examples \ref{ex:1}, \ref{ex:p} and \ref{ex:p2} give rise to irreducible, hence smooth,  curves. 
We are able to settle this point at least in a special case:
\begin{prop} \label{prop:pd-irred}
Let $n\ge 2,  d\ge 3 $ be integers; let $E$ be an elliptic  curve without complex multiplication, $\xi\in E$ an element of order $r:=dn-1$, and $\phi'\colon E\to E':=E/<\xi>$ the quotient map. 

Then the $(n,d)$-elliptic genus 2 curve  constructed as in Example \ref{ex:1} with   $F=E$, $\phi={\rm Id}_{E}$ and $\phi'$ as above  is smooth.
\end{prop}

As an immediate consequence we obtain:

\begin{cor}
For every pair of integers $n, d>1$ there exists a smooth $(n,d)$-elliptic curve of genus 2 with twisting number $m=1$.
\end{cor}
\begin{proof}
For  $n=d=2$ the claim follows by Lemma \ref{lem:bi-red}, for instance by using the construction of Example \ref{ex:1},    and by Proposition \ref{prop:pd-irred} in  the remaining cases.
\end{proof}
\begin{proof}[Proof of Proposition \ref{prop:pd-irred}] 
Denote by $\Xi$ the product polarization on $E\times E'$. Set $H:=\{(\eta,\phi'(\eta))|\eta\in E[p]\}$  and let $h\colon E\times E'\to A:=(E\times E')/H$ be the quotient map. 

We argue  by contradiction, so assume that the principal polarization  of $A$ induced by $n\Xi$ is reducible and denote it by $C=C_1+C_2$. Up to a translation we may assume that  the singular point of $C$ is the origin of  $A$. Let  $\wt C_i$ be  the connected component  of the preimage of $C_i$ containing the origin of $E\times E'$, $i=1,2$, so that $h^*C_i$ is numerically equivalent to $\nu_i \wt C_i$ for a positive  integer $\nu_i$. One has 
\begin{equation}\label{eq:n}
n^2=\nu_i|H\cap \wt C_i|\  \text{ and }\ \frac{n}{\nu_i}=\wt C_i\Xi\in \IZ.
\end{equation}

Since $E$ does not have complex multiplication ($\End E  = \IZ$), the connected  $1$-dimensional subgroups of $E\times E$ distinct from $E\times\{0\}$ and $\{0\}\times E$ are of the form $\{(ax,bx)\  |\  x\in E\}$, with $a, b$ coprime integers. This is well known, but we give a quick proof for lack of a suitable reference. 
Let $G$ be such a subgroup, and denote by $\psi_i\colon G\to E$, $i=1,2$ the isogenies induced by the two projections. Note that $\ker\psi_1\cap \ker \psi_2=\{0\}$. If $G$ is isomorphic to $E$, then the $\psi_i$ are multiplication maps and $G$ is of the form $D_{a,b}$ for some pair of coprime integers $a,b$.
So assume that $G$ and $E$ are not isomorphic and consider an isogeny $\chi\colon E\to G$. Since $\chi$ is not a multiplication map, there exists an integer $k$ and elements $u, v\in E[k]$ such that $\chi(u)=0$ and $\chi(v)=v'\ne 0$. Now consider the maps $\mu_i:=\psi_i\circ \chi\colon E\to E$, which are multiplication maps by  integer $t_i$, $i=1,2$. Both $t_1$ and $t_2$ are divisible by $k$, since for $i=1,2$ we have $\mu_i(u)=0$, hence $\psi_i(v')=\mu_i(v)=0$  and so $v'=0$, a contradiction. 

 Since $D_{a,b}=D_{-a,-b}$, we  may always assume $a\geq 0$. 
 It follows that the connected  $1$-dimensional subgroups of $E\times E'$ distinct from $E\times\{0\}$ and $\{0\}\times E'$ are of the form  $D_{a,b}=\{(ax,b\phi'(x))\ |\ x\in E\}$.
Notice that the kernel of the induced  map $E\to D_{a,b}$ is the cyclic subgroup of $<\xi>$ of order $\delta:=g.c.d.(a,r)$.
Using this observation one computes:
\begin{equation}\label{eq:intersections}
D_{a,b}(\{0\}\times E')=\frac{a^2}{\delta}, \ \ D_{a,b}( E\times \{0\})=\frac{b^2r}{\delta},  \ \  D_{a,b} D_{1,1}=(b-a)^2\frac{r}{\delta}.
\end{equation}
For $i=1,2$, let $a_i, b_i\in \IZ$ be such that $\wt C_i=D_{a_i,b_i}$,  with  $a_i\ge0$; set $\delta_i=g.c.d.(a_i,r)$.  We will now derive a contradiction using intersection numbers.

\begin{description}
 \item[Step 1] 
 We have $a_i>0$. Indeed  if $a_i=0$ we have  $\wt C_i \Xi=1$ and  $ |H\cap \wt C_i| =1$, so \eqref{eq:n} gives $n=\nu_i$ and $\nu_i=n^2$,    against our assumptions. 
\item[Step 2] 
  We show $(a_i, b_i)\neq (1,1)$. Indeed, assuming $\tilde C_i = D_{1,1} $ \eqref{eq:n} gives
 \[\frac{n}{\nu_i} = \tilde C_i\Xi =  D_{1,1}\Xi = 1+r = dn,\]
 which is impossible since $d>1$. In particular, since  $a_i$ and $b_i$ are coprime, we have $a_i\neq b_i$. 
\item[Step 3] From the above steps  we derive  two inequalities and a divisibility property which will lead to a contradiction.
 
 First of all we
 have, for $i=1,2$,
 \begin{equation}\label{eq: divisibility}
  n | \nu_i (a_i - b_i),
 \end{equation} 
 Indeed, since   $D_{1,1}\cap  \wt C_i$ is a subgroup containing $H\cap \wt C_i$ we have that 
$\wt C_i D_{1,1}=(b_i-a_i)^2\frac{r}{\delta_i}$ is divisible by $\frac{n^2}{\nu_i}$, hence $(b_i-a_i)^2$ is divisible by $\frac{n^2}{\nu_i}$, since $\frac{r}{\delta_i}$ is an integer prime to $n$. So  $\nu_i^2(b_i-a_i)^2$ is divisible by $n^2$, and therefore  $\nu_i(a_i-b_i)$ is divisible by $n$. 
 
Secondly, by \eqref{eq:intersections} we have $n=(\nu_1\wt C_1+\nu_2\wt C_2)(\{0\}\times E')=\nu_1a_1\frac{a_1}{\delta_1}+\nu_2a_2\frac{a_2}{\delta_2}$ and $n=(\nu_1\wt C_1+\nu_2\wt C_2)(E\times \{0\})=\nu_1b_1^2\frac{r}{\delta_1}+\nu_2b_2^2\frac{r}{\delta_2}$. In particular, we have
 \begin{equation}\label{eq:ineq}
 \nu_1 a_1+\nu_2a_2\le n, \qquad d(\nu_1b_1^2+\nu_2 b_2^2)\le n, 
  \end{equation}
  since $\frac{r}{\delta_i}$ is an integer and $\frac{r}{\delta_i}\ge d\frac{n}{a_i}-1>d-1$.
 \item[Step 4] We cannot have $b_i>0$. Indeed in this case, since $0\le \nu_i a_i, \nu_i b_i <n$ by \eqref{eq:ineq} and $n$ divides the difference  $\nu_i a_i- \nu_i b_i$ by \eqref{eq: divisibility}, then  we necessarily have $a_i = b_i$ contradicting Step 2.
 \item[Step 5] We cannot have $b_i \leq 0$. Indeed the same argument as in the previous step shows that we would necessarily have $\nu_ib_i=\nu_ia_i-n$ for $i=1,2$.  By \eqref{eq:ineq} we may assume, say, $\nu_1a_1\le \frac{n}{2}$ and thus by the above equality $\nu_1|b_1|=-\nu_1b_1\ge \frac{n}{2}$. Then \eqref{eq:ineq} gives:
$$n\ge d\nu_1b_1^2 \ge  | b_1|  \frac{dn}{2},$$
 a contradiction since $d>2$. 
\end{description}
Combining the last two steps we arrive at a contradiction and have thus proved that that the polarisation is irreducible and hence is  a smooth $(n,d)$-elliptic curve of genus $2$.
\end{proof}

\section{bi-tri-elliptic curves}\label{section:bi-tri}

For the applications to the classification of Gorenstein stable Godeaux surfaces the case of bi-tri-elliptic configurations is of particular interest. In this section we first formulate Theorem \ref{thm:2d} in  this case and then analyse reducible bi-tri-elliptic curves in more detail.

Indeed we have the following characterization of  reducible bi-tri-elliptic configurations. 
\begin{cor}\label{cor:bi/tri} Let $(C, f,g)$ 
 be a bi-tri-elliptic configuration on a stable curve of arithmetic genus $2$. Then the twisting number $m$ defined in \eqref{eq: m} satisfies $1\le m\le 5$ and there are the following possibilities:
\begin{itemize}
\item[(a)]  $m$ is  odd and  the configuration arises as in Example \ref{ex:1} with $\deg\phi=m$;
\item [(b)]$m=2\mu$ is even and  the configuration  arises as in Example \ref{ex:p} with $\deg\phi=\mu$.
\end{itemize}
\end{cor}

\begin{rem}
 Counting parameters we see that the space of bi-tri-elliptic configurations is one-dimensional, but we did not consider its finer structure, e.g., the number of irreducible or connected components.
\end{rem}
 
Now assume that 
 $C\cong E \cup_0 E$,  where $E$ is  an elliptic curve with a  degree $2$ endomorphism $\psi\colon E\to E$.
 Then we can build a  {  natural}  bi-tri-elliptic  configuration 
\begin{equation}\label{diagram:bi-tri-configuration}
 \begin{tikzcd}{} &E \cup_0 E\arrow{dl}{2:1}[swap]{f=\id\cup\id}\arrow{dr}{g=\id\cup \psi}[swap]{3:1}\\
   E && E
  \end{tikzcd}.
 \end{equation}
We will now show that every bi-tri-elliptic configuration $(C, f, g)$ with $C$ reducible is of this form. Indeed, by Lemma \ref{lem:bi-red} the bi-elliptic map $f$  on the reducible curve $C$ is isomorphic to the composition of horizontal arrows in the diagram 
 \[
 \begin{tikzcd}
 {}& \bar F\dar\\
 C = E\times\{0\}\cup\{0\}\times E \arrow{dr}{g}\arrow[hookrightarrow]{r} &   E\times E \rar{+}\dar & E\\
  & D
 \end{tikzcd},
\]
and the tri-elliptic map is uniquely determined by the subgroup $\bar F$. Note that the covering involution of $f$ exchanges the components of $C$. 

We have $\bar F C =3$ and without changing $f$ we can assume that $\bar F(\{0\}\times E)=1$ and $\bar F( E\times\{0\})=2$. In other words, $\bar F$ is the graph of a  degree $2$ endomorphism $\psi\colon E\to E$ and $E\times\{0\}$ is identified with the second elliptic curve $D$ by the restriction of $g$. 
Therefore the bi-tri-elliptic configurations  is as in \eqref{diagram:bi-tri-configuration}.

An isomorphism from $(C, f,g)$ to another bi-tri-elliptic configuration $(\tilde C = \tilde E \cup_0\tilde E, \tilde f, \tilde g)$ such that $\tilde f$ is the identity on each component of $\tilde C$ is uniquely determined by an isomorphism $E\isom \tilde  E$ and thus we have proved the first part of the following 
\begin{prop}
 The above construction induces a bijection on the set of iso\-morphism-classes of bi-tri-elliptic configurations $(C, f,g)$ with $C$ a reducible stable curve of genus $2$ and the set $\{(E, \psi)\}$ of elliptic curves together with an endomorphism of degree $2$. 

For every $1\leq m\leq 5$ there are exactly two such pairs $(E,\psi)$, which are listed in Table \ref{tab: endos}, thus in total there are 10 isomorphism classes of bi-tri-elliptic configurations with $C$ a reducible stable curve of  genus $2$. 
\end{prop}

\begin{table}[ht]\caption{Endomorphisms  of degree 2  on elliptic curves}\label{tab: endos}
 \begin{center}
 \begin{tabular}{llcc}
 \toprule
$E=\IC/\Gamma$ & $\Gamma = \End(E)$ & $\xi$& $m$\\
\midrule
\multirow{2}*{$E_1$} &  \multirow{2}*{$\IZ\left[\I\right]$} & $ -1\pm \I$ & $1$\\
			  & & $ \phantom{-}1\pm \I$ & $5$\\
			  \midrule
$E_2$ & $\IZ\left[\I\sqrt{2}\right]$& $\pm \I\sqrt{2}$ & $3$\\
\midrule
\multirow{2}*{$E_3$}&  \multirow{2}*{$\IZ\left[\frac12(1+\I\sqrt{7})\right]$} & 
$ -\frac12( 1\pm \I\sqrt{7})$ & $ 2$\\
&&$ \phantom{-}\frac12(1\pm\I\sqrt{7})$ & $4$\\
\bottomrule
\end{tabular}
\end{center}
\end{table}

\begin{proof}
We need to recall some elementary facts about endomorphisms of elliptic curves. Details can be found for example in \cite[Ch.11]{silverman-Arithmetic} or \cite[Ch.II]{silverman-ArithmeticII}.
Any endomorphism $\psi$ of an elliptic curve $E$ is given by multiplication by a complex number $\xi$ and this embeds $\End E \into \IC$ as a maximal order in an imaginary quadratic number field $K\isom \End(E)\tensor \IQ$. The degree of the endomorphism $\psi$ coincides with the norm $N_{K/\IQ}(\xi)$.

Thus elements inducing an endomorphism of degree $2$ are  characterised as those $\xi \in \IC\setminus \IR$ that are integral over $\IZ$ with  characteristic polynomial
\[ p_\xi (t) = t^2-\mathrm{trace}_{K/\IQ}(\xi)t + N_{K/\IQ}(\xi) = t^2-2\mathrm{Re}(\xi) t+2 \in \IZ[t].\]
This gives exactly the elements listed in Table \ref{tab: endos} and each one of them is contained in a unique maximal order by \cite[Example 11.3.1]{silverman-Arithmetic} (see also 
 \cite[Prop. 2.3.1]{{silverman-ArithmeticII}}).

It remains to compute the invariant $m$, which is in our case the intersection of $ \Gamma_\psi=\bar F \subset E\times E$  with  the kernel of the addition map, that is, the anti-diagonal. Thus $m$ equals the number of fixed points of the endomorphism 
$-\psi$, which by the holomorphic Lefschetz fixed-point formula \cite[Ch.~3.4]{Griffiths-Harris} gives
\[m = \sum_{i=0}^2 (-1)^i\mathrm{trace}\left(-\psi_* |_{H_i(E, \IQ)}\right)= 1-\mathrm{trace}_{K/\IQ}(-\xi) + N_{K/\IQ}(-\xi)=p_\xi(-1), \]
because every fixed point of $\psi$ is simple.
\end{proof}
%
 \def\cprime{$'$}

\end{document}